\documentclass[a4paper,10pt]{amsart}
\usepackage{graphicx, amssymb, amsmath, amscd}
\usepackage[arrow,tips,matrix]{xy}
\theoremstyle{plain}
\newtheorem{theorem}{Theorem}
\newtheorem{proposition}{Proposition}

\newtheorem*{main theorem}{Theorem}
\newtheorem{question}[theorem]{Question}
\newtheorem{corollary}[theorem]{Corollary}

\theoremstyle{definition}

\newtheorem{remark}{Remark}

\newtheorem{step}{Step}
\newtheorem*{properties}{Properties of $(X,\omega)$}

\numberwithin{table}{section} \numberwithin{figure}{section}
\numberwithin{equation}{section}

\DeclareMathOperator{\im}{im}

\DeclareMathOperator{\Hom}{Hom} \DeclareMathOperator{\Crit}{Crit}

\newcommand{\field}[1]{\mathbb{#1}}
\newcommand{\C}{\field{C}}
\newcommand{\R}{\field{R}}

\newcommand{\fiber}{\mathrm{f}}

\begin{document}

\title[An example of circle actions on symplectic Calabi-Yau manifolds]{An example of circle actions on symplectic Calabi-Yau manifolds with non-empty fixed points.}

\author[Y. Cho]{Yunhyung Cho}
\address{School of Mathematics, Korea Institute for Advanced Study,
87 Hoegiro, Dongdaemun-gu, Seoul, 130-722, Republic of Korea}
\email{yhcho@kias.re.kr}

\author[M. K. Kim]{Min Kyu Kim*}
\address{Department of Mathematics Education,
Gyeongin National University of Education, 45 Gyodae-Gil,
Gyeyang-gu, Incheon, 407-753, Republic of Korea}
\email{mkkim@kias.re.kr}

\thanks{* the corresponding author supported by GINUE research fund}

%

\begin{abstract}
Let $(X,\sigma,J)$ be a compact K\"{a}hler Calabi-Yau manifold
equipped with a symplectic circle action. By
Frankel's theorem \cite{F}, the action on $X$ is non-Hamiltonian and $X$ does not have any
fixed point. In this paper, we will show that a symplectic circle
action on a compact non-K\"{a}hler symplectic Calabi-Yau manifold
may have a fixed point. More precisely, we will show that the
symplectic $S^1$-manifold constructed by D. McDuff \cite{McD} has
the vanishing first Chern class. This manifold has the Betti numbers $b_1 = 3$, $b_2 = 8$, and $b_3 = 12$. In particular, it does not admit any K\"{a}hler structure.
\end{abstract}

\maketitle

\section{Introduction}\label{1}
Let $(X,\omega)$ be a symplectic manifold and let $G$ be a Lie group
acting on $X$. We say that the $G$-action is \textit{symplectic} if
$g^* \omega = \omega$ for every $g \in G$. Equivalently, the action
is symplectic if and only if $i_{\xi} \omega$ is a closed 1-form for
every Lie algebra element $\xi \in \mathfrak{g}$, where $i_{\xi}
\omega = \omega(\underline{\xi},\cdot)$ is an interior product with
the fundamental vector field $\underline{\xi}$ of $\xi$. When
$i_{\xi} \omega$ is exact for every $\xi \in \mathfrak{g}$, then we
say that the $G$-action is \textit{Hamiltonian}. If the action is
Hamitonian, there exists a map $\mu : X \rightarrow \mathfrak{g}^*$
defined by $$ \mu(p)(\xi) := \mu_{\xi}(p), \mathrm{\forall p \in
X}, \mathrm{\forall \xi \in \mathfrak{g}},$$ where $\mu_{\xi} : X \rightarrow \R$ is a $C^\infty$-function
on $X$ such that $d\mu_{\xi} = i_{\xi}\omega$. We call $\mu$ a
\textit{moment map} for the $G$-action. Therefore, it is natural to
ask the followings.

\begin{question}\label{Q}
Let $(X,\omega)$ be a compact symplectic manifold. Then what
conditions on $(X,\omega)$ make the symplectic action to be
Hamiltonian? (non-Hamiltonian, respectively)
\end{question}
The following results are related to the question \ref{Q}.
\begin{theorem}\label{result}\cite{CKS, F, LO, McD, O}
Let $(X,\omega)$ be a compact symplectic manifold.
\begin{enumerate}
 \item If $X$ is simply-connected, then any symplectic action is
Hamiltonian.
 \item If $(X,\omega,J)$ is a K\"{a}hler manifold and if a given
symplectic circle action preserves $J$, then the action is
Hamiltonian if and only if the fixed point set $X^{S^1}$ is
non-empty. (T. ~Frankel ~1959 ~\cite{F}.)
 \item If $\dim(X) =4$, then any symplectic circle action
is Hamiltonian if and only if the fixed point set $X^{S^1}$ is
non-empty. (D. McDuff 1988 \cite{McD}.)
 \item If $(X,\omega)$ satisfies the weak Lefschetz property,
then any symplectic circle action is Hamiltonian if and only if the
fixed point set $X^{S^1}$ is non-empty. (K. Ono 1988 \cite{O}.)
 \item If $(X,\omega)$ is a monotone symplectic manifold,
then any symplectic circle action is Hamiltonian. (G. Lupton, J.
Oprea 1995 \cite{LO}.)
 \item If $(X,\omega)$ is symplectic Calabi-Yau,
i.e. $c_1(X)=0$ in $H^2(X;\R)$, then any symplectic circle action is
non-Hamiltonian.(Y. Cho, M. K. Kim, D. Y. Suh 2012 \cite{CKS}.)
\end{enumerate}
\end{theorem}
In particular, if we combine the results (2) and (6) of Theorem \ref{result}, then we have
the following corollary.

\begin{corollary}\label{CY}
Let $(X,\omega,J)$ be a compact K\"{a}hler Calabi-Yau manifold.
Assume that there is a symplectic circle action preserving $J$.
Then $X$ has no fixed point.
In particular, the automorphism group of any simply connected K\"{a}hler Calabi-Yau manifold is discrete.
\end{corollary}
The main purpose of this paper is to announce that Corollary
\ref{CY} does not hold on symplectic Calabi-Yau manifolds in
general. Here, we state our main theorem.

\begin{theorem} \label{theorem: main}
There exists a compact symplectic Calabi-Yau manifold $(X,\omega)$
equipped with a symplectic circle action such that the fixed point
set $X^{S^1}$ is non-empty.
\end{theorem}

In fact, there exists a 6-dimensional compact symplectic manifold
equipped with a symplectic non-Hamiltonian circle action with
non-empty fixed point set, which is constructed by D. McDuff in
\cite{McD}. As far as the authors know, McDuff's example is the only
one well-known example of symplectic non-Hamiltonian $S^1$-manifold
with non-empty fixed point set. From now on, we denote the McDuff's
manifold by $(W,\widetilde{\omega})$. As we will see in Section
\ref{section: McDuff's example}, $(W,\widetilde{\omega})$ can be
obtained by the quotient of some Hamiltonian $S^1$-manifold
$(X,\omega)$ with the moment map $\mu : X \rightarrow [0,7] \subset
\R$ with two boundaries $\mu^{-1}(0)$ and $\mu^{-1}(7),$ where the
quotient map is given by some $S^1$- ~equivariant diffeomorphism
between $\mu^{-1}(0)$ and $\mu^{-1}(7)$. From Section
\ref{section: level set} to Section \ref{section: homology of W}, we
give an explicit computation of $H_*(W;\R)$ by
following steps below.

\begin{step}\label{step1}
According to \cite{McD}, the critical values of the moment map $\mu : X \rightarrow [0,7]$
are $1,2,5$,and $6$. Hence the set of regular values of $\mu$ is
$$[0,1) ~\cup (1,2) ~\cup (2,5) ~\cup (5,6) ~\cup (6,7].$$ Choose
one value for each connected open regular interval. (In this paper,
we will choose $t_1 = 0 \in [0,1)$, $t_2 = 1.5 \in (1,2)$, $t_3 = 3.5 \in
(2,5)$, and $t_4 = 5.5 \in (5,6)$. And then, we compute
$H_1(\mu^{-1}(t);\R)$ and $H_2(\mu^{-1}(t);\R)$ for each regular value $t =
t_1, \cdots, t_4$,  and describe the
generators of $H_1(\mu^{-1}(t);\R)$ and $H_2(\mu^{-1}(t);\R)$ by some submanifolds of $\mu^{-1}(t)$.
Step \ref{step1} will be discussed in Section
\ref{section: level set}, and we use the result of \textbf{Step
\ref{step1}} in the remaining steps.
\end{step}

\begin{step}\label{step2}
Note that the moment map $\mu$ is a Morse-Bott function, so we can express
$X$ as the union of ``\textit{elementary cobordisms}" as follows.
$$ X = X_1 \cup X_2 \cup X_3 \cup X_4, $$
where $X_1 := \mu^{-1}([0, 1.5])$, $ X_2 := \mu^{-1}([1.5, 3.5])$,
$X_3 := \mu^{-1}([3.5, 5.5])$, and $X_4 := \mu^{-1}([5.5, 7]). $ In
Section \ref{section: elementary cobordism}, we will study the
topology of each elementary cobordism $X_i$ to compute $H_1(X_i;\R)$
and $H_2(X_i;\R)$ for $i=1, \cdots, 4$.
\end{step}
\begin{step}\label{step3}
For each $i=1, \cdots, 4$, we will compute $H_1(X_i;\R)$ and
describe the generators of $H_1(X_i;\R)$ in Section \ref{section:
H_1 of elementary cobordism}. In Section \ref{section: H_2 of
elementary cobordism}, we will compute $H_2(X_i;\R)$ and describe
the generators of $H_2(X_i;\R)$ by some 2-dimensional
symplectic submanifolds of $X$ for $i=1, \cdots, 4$. After then, we
compute $H_j(X_1 \cup X_2;\R)$ and $H_j(X_3 \cup X_4;\R)$ for $j=1$
or $2$ in Section \ref{section: homology of [0, 3.5]}. Finally, we
will show the followings in Section \ref{section: homology of W}:
\begin{itemize}
\item $H_1(W;\R) \cong \R^3,$
\item $H_2(W;\R) \cong \R^8,$
\item $\langle c_1(W), A \rangle = 0$ for every generator $A \in H^2(W;\R)$.
\end{itemize}
\end{step}

\begin{remark}
Using the fixed point formula for Hirzebruch genus, the Euler characteristic of $W$ equals to the sum of the Euler characteristic of the fixed components. So one can easily deduce that the Euler characteristic $\chi(W) = 0$ which is the alternating sum of the Betti numbers. Hence we have $H_3(W;\R) = \R^{12}$.
\end{remark}

\begin{remark}
Note that $b_1(W) = 3$. Therefore, $W$ does not admit any K\"{a}hler
structure (since $b_1$ is odd).
\end{remark}

\section{McDuff's example} \label{section: McDuff's example}
In this section, we review the 6-dimensional semifree symplectic
non-Hamiltonian $S^1$-action $(W,\widetilde{\omega})$ constructed by
McDuff \cite{McD}, and then define some submanifolds of $W$ by which
we will describe $H_1 (W; \R)$ and $H_2 (W; \R)$ in the below.
Henceforward, the coefficient of homology and cohomology groups is
$\R,$ and we use integral indices $1 \le i, j \le 4,$ and if they are
used simultaneously, then $i \ne j.$

As we introduced in Section \ref{1}, $(W,\widetilde{\omega})$ can be
obtained by the quotient of some Hamiltonian $S^1$-manifold
$(X,\omega)$ with the moment map $\mu : X \rightarrow [0,7] \subset
\R$ with two boundary components $\mu^{-1}(0)$ and $\mu^{-1}(7)$ such that
$$ W = X / \sim \hspace{0.1cm}, \hspace{0.2cm} x \sim y \Leftrightarrow \{x \in \mu^{-1}(0), \hspace{0.05cm}y \in \mu^{-1}(7),\hspace{0.05cm} \tau(x) = y\}$$
for some $S^1$-equivariant diffeomorphism $\tau : \mu^{-1}(0) \rightarrow \mu^{-1}(7).$ We summarize properties of $(X, \omega)$ as follows:
\begin{properties}\label{property}\cite{McD} $(X,\omega)$ satisfies the followings.
\begin{enumerate}
  \item $X$ has four critical levels at $\lambda=1, 2, 5, 6,$ and the
  critical point set is the union of four two-tori $Z^\lambda$ for
  $\lambda=1, 2, 5, 6$ such that $\mu(Z^\lambda)=\lambda.$

\item Let $B = (S^1)^4$ be a 4-dimensional torus with coordinates $(x^1, x^2, x^3, x^4).$ By \cite[Lemma 5.(i)]{McD}, there exists an $S^1$-invariant
  smooth map $\hat{\pi} : X \rightarrow B$ satisfying the followings.
  \begin{enumerate}
\item For each regular value $s$, the restriction map $\hat{\pi}|_{\mu^{-1}(s)} : \mu^{-1}(s) \rightarrow B$ induces a diffeomorphism $\mu^{-1}(s)/S^1 \cong B.$
\item Let $\xi$ be the fundamental vector field of the $S^1$-action
   and let $J$ be an $S^1$-invariant $\omega$-compatible almost complex structure.
  If two points $x, x^\prime \in X$ are connected by a flow
    $J \xi,$ then $\hat{\pi}(x)=\hat{\pi}(x^\prime).$
    \item Let $L_{ij} \subset B$ a two-torus on which the two coordinates
     other than $x^i$ and $x^j$ are constant. Then
    \begin{equation*}
    \hat{\pi}(Z^\lambda) = \left\{       \begin{array}{ll}
                               L_{13} & \quad \hbox{for } s=1,5, \\
                               L_{24} & \quad \hbox{for } s=2,6.
                             \end{array}
                           \right.
    \end{equation*}
    To avoid confusion, we denote $Z^\lambda$ by $Z_{ij}^\lambda$
    when $\hat{\pi}(Z^\lambda)=L_{ij}.$
  \end{enumerate}
  \item
  Let $\sigma_{ij} = dx^i \wedge  dx^j$ and $\sigma_i = dx^i$ on $B$.
  For any regular value $s,$ the Chern classes of the principal $S^1$-bundle
  $\hat{\pi}|_{\mu^{-1}(s)} : \mu^{-1}(s) \rightarrow B$
  is as follows.
  \begin{equation*}
  c_1 (\hat{\pi}|_{\mu^{-1}(s)}) = \left\{
             \begin{array}{ll}
               0                          & \quad \hbox{for } s \in [0,1),  \\
               -[\sigma_{42}]             & \quad \hbox{for } s \in (1,2),  \\
               -[\sigma_{31}+\sigma_{42}] & \quad \hbox{for } s \in (2,5),  \\
               -[\sigma_{31}]             & \quad \hbox{for } s \in (5,6),  \\
               0                          & \quad \hbox{for } s \in (6,7].
             \end{array}
           \right.
  \end{equation*}

  \item For $s=0, 7,$ we may consider $\mu^{-1}(s)$ as $T^5$ (since $c_1 (\hat{\pi}|_{\mu^{-1}(s)})=0)$ with coordinates
  $x^1, x^2, x^3, x^4, x^5$ so that $\hat{\pi}|_{\mu^{-1}(s)}$
  is expressed as
\begin{equation*}
\hspace{1cm} \hat{\pi}|_{\mu^{-1}(s)} : \mu^{-1}(s) = T^5 \longrightarrow B =
T^4, \quad (x^1, x^2, x^3, x^4, x^5) \mapsto (x^1, x^2, x^3, x^4),
\end{equation*}
Obviously, the circle action on $\mu^{-1}(s)$ is given by $$t \cdot (x^1, x^2, x^3, x^4, x^5) = (x^1, x^2, x^3, x^4, t x^5)$$ for every $t \in S^1$.
\end{enumerate}
\end{properties}
Our manifold $(W,\widetilde{\omega})$ is obtained by gluing $\mu^{-1}(0)$ to $\mu^{-1}(7)$, where the gluing map $\tau : \mu^{-1}(0) \rightarrow \mu^{-1}(7)$ is given by an involution
\begin{equation*}
\tau : \mu^{-1}(0) \rightarrow \mu^{-1}(7), \quad (x^1, x^2, x^3,
x^4, x^5) \mapsto (x^3, x^4, x^1, x^2, x^5).
\end{equation*}
Let $\tilde{\pi} : X \rightarrow W$ be the quotient map induced by the gluing. Then, $W$
carries the generalized moment map
\begin{equation*}
\psi : W \rightarrow S^1, \quad \tilde{\pi}(x) \longmapsto \exp
\Big( \frac {2\pi \imath ~ \mu(x)} 7 \Big), \qquad \forall{x} \in
X.
\end{equation*}

To describe the generators of $H_1 (W)$ and $H_2 (W)$ explicitly, we define
some submanifolds of $W.$ First, we define loops in $W.$ For $s=0$ or $7$, denote by
$L_i^s \subset \mu^{-1}(s) = T^5$ a circle on which coordinates other than $x^i$ are constant  $1 \le i \le 4,$ and
denote by $L_\fiber^s$ a circle in $\mu^{-1}(s) = T^5$ on which
coordinates other than $x^5$ are constant. Here, $\fiber$ is the
initial alphabet of ``fiber".
\begin{remark}
    Sometimes, we will describe the generators of $H_*(W)$ by submanifolds of $X$, i.e. the reader should regard any submanifold $K \subset X$ as an    image $\tilde{\pi}(K) \subset W$. For example, the gluing map $\tau$ identifies $L_1^0$ with $L_3^7$ and identifies $L_2^0$ with $L_4^7$. Hence we should keep in mind that $L_1^0 = L_3^7$ and $L_2^0 = L_4^7$ as generators of $H_1(W).$
\end{remark}
We define one more loop in $W.$ Let $\Crit \psi$ be the set of all critical points of $\psi$, i.e. $\Crit \psi = \cup_{\lambda = 1,2,5,6} Z^{\lambda}.$  Fix two points $y_0,$ $y_1 \in W$ in the
level set $\psi^{-1} (1),$ $\psi^{-1} (\imath),$ respectively. And,
choose two paths $\gamma_1, \gamma_2 : [0,1] \rightarrow W-\Crit \psi$
such that
\begin{equation*}
\begin{array}{ll}
\gamma_1 (0) = \gamma_2(1)= y_0, \quad & \gamma_1 (1) = \gamma_2(0)= y_1,  \\
\psi(\gamma_1 (t)) = \exp(\pi \imath t), \quad & \psi(\gamma_2 (t))
= \exp(\pi \imath (t+1)).
\end{array}
\end{equation*}
Let $\gamma : [0,1] \rightarrow W$ be the loop $\gamma_1.\gamma_2$ so that $\psi|_{\gamma} : \gamma \rightarrow S^1$ is a bijection.
In Section ~\ref{section: homology of W}, we will prove
\begin{proposition} \label{proposition: H_1}
$H_1 (W) = \big\langle L_1^0, L_2^0, \gamma \big\rangle \cong \R^3.$
\end{proposition}

Second, we define some tori and a sphere in $W.$  For $s=0$ or $7$, denote by
$L_{ij}^s \subset \mu^{-1}(s) = T^5$ a two-torus on which
coordinates other than $x^i$ and $x^j$ are constant, and denote by
$L_{i \fiber}^s$ a two-torus in $\mu^{-1}(s) = T^5$ on which the two coordinates other than $x^i$ and $x^5$ are
constant for $1 \le i,j \le 4.$ Also, let $G_{61} \subset W$ be an $S^1$-invariant sphere which connects $Z_{24}^6$ and $Z_{13}^1$ and
whose image through $\hat{\mu}$ is the counterclockwise arc from
$\exp(\frac{2\pi \cdot 6}{7})$ to $\exp(\frac{2\pi \cdot 1}{7}).$
\begin{remark}
    Note that $G_{61}$ can be chosen as follows. Let $\gamma : [0,1] \rightarrow W$ be a path such that
    \begin{itemize}
        \item $\gamma(0) \in Z_{24}^6$,
        \item $\gamma(1) \in Z_{13}^1$, and
        \item $\langle \mathrm{grad}\gamma, J\xi \rangle > 0$, where $\langle , \rangle$ is a Riemannian metric induced by $\omega$ and $J$.
    \end{itemize}
Then we get a 2-sphere $G_{61} := S^1 \cdot \gamma$.
\end{remark}

Now, we define two more tori. Let $L_{1+3},$ $L_{2+4}$ be loops in $\mu^{-1}(0)$
defined by
\begin{align*}
& \{(\exp(2\pi t \imath), 1, \exp(2\pi t \imath), 1, 1) ~|~ t \in \R
\}, \\
& \{(1, \exp(2\pi t \imath), 1, \exp(2\pi t \imath), 1) ~|~ t \in \R
\},
\end{align*}
respectively. And, let $L_{1+3}^\prime,$ $L_{2+4}^\prime$ be loops
in $\mu^{-1}(7)$ defined by
\begin{align*}
& \{(\exp(2\pi t \imath), 1, \exp(2\pi t \imath), 1, 1) ~|~ t \in \R
\}, \\
& \{(1, \exp(2\pi t \imath), 1, \exp(2\pi t \imath), 1) ~|~ t \in \R
\},
\end{align*}
respectively. By varying constant coordinates of $L_{ij}$ other than
$x_i, x_j,$ we may assume that $\hat{\pi}(L_{1+3}),$
$\hat{\pi}(L_{2+4})$ do not intersect any $L_{ij}$ in $B.$ By
(2)-(b) of ``properties of $(X,\omega)$", this implies that we may
assume that $L_{1+3}$ and $L_{1+3}^\prime$ (also $L_{2+4}$ and
$L_{2+4}^\prime$) are connected by gradient flow of $J\xi.$ And the
trajectory of $L_{1+3}$ along the gradient flow of $J\xi$ swept by
the gradient flow beginning at $L_{1+3}$ and $L_{2+4}$ are
diffeomorphic to $L_{1+3} \times [0, 7]$ and $L_{2+4} \times [0, 7]$
whose images $\tilde{\pi}(L_{1+3} \times [0, 7])$ and
$\tilde{\pi}(L_{2+4} \times [0, 7])$ are called $T_{1+3}^2$ and
$T_{2+4}^2,$ respectively. In Section \ref{section: homology of W},
we will prove the followings.
\begin{proposition} \label{proposition: H_2}
\begin{equation*}
H^2(W) = \langle ~ L_{12}^0, ~ L_{13}^0, ~ L_{14}^0, ~ L_{24}^0, ~
Z_{24}^2 \text{ or } Z_{13}^5, ~ T_{1+3}^2, ~ T_{2+4}^2, ~ G_{61} ~
\rangle \cong \R^8.
\end{equation*}
\end{proposition}

\begin{proposition} \label{proposition: calabi-yau}
$c_1 ( T W|_Q)=0$ for each generator $Q$ of $H^2(W).$
\end{proposition}

\section{Homology groups of a regular level set $\mu^{-1}(s)$}
\label{section: level set}

In this section, we define $L_i^s,$ $L_{i \fiber}^s,$ $L_{ij}^s$ in
$\mu^{-1}(s)$ for $s=1.5,$ $3.5,$ $5.5$ as we have defined for $s=0,
7,$ and we compute the homology groups $H_i(\mu^{-1}) (s)$
for regular values $s=0,$ $1.5,$ $3.5$ for $i=1,2$.

Let $\pi : P \rightarrow B$ be an oriented smooth $S^1$-bundle.
Then, the Gysin sequence is the following long exact sequence of de
Rham cohomology groups:
\begin{equation} \label{sequence: gysin}
\footnotesize \SelectTips{cm}{} \xymatrix{ H^p (B) \ar[r]^-{e
\wedge} & H^{p+2} (B) \ar[r]^-{\pi^*} & H^{p+2} (P) \ar[r]^-{\pi_*}
& H^{p+1} (B) \ar[r]^-{e \wedge}  & H^{p+3} (B) \ar[r] & }
\end{equation}
where $\pi^*$ is the pullback induced by $\pi,$ $e \wedge$ is the
wedge product of a differential form with the Euler class $e$ of the
bundle, and $\pi_*$ is the integration along the fiber of
differential forms. Applying $\Hom_\R(\cdot, \R)$ to this, we obtain
the following exact sequence:
\begin{equation} \label{sequence: homological gysin}
\footnotesize \SelectTips{cm}{} \xymatrix{   H^p (B)  & \ar[l]_-{(e
\wedge)^\prime} H^{p+2} (B) & \ar[l]_-{(\pi^*)^\prime} H^{p+2} (P) &
\ar[l]_-{\pi_*^\prime} H^{p+1} (B)  & \ar[l]_-{(e \wedge)^\prime}
H^{p+3} (B) & \ar[l]}
\end{equation}
where $^\prime$ means transpose. Here, $(\pi^*)^\prime$ is equal to
the homology functor $H_1(\pi).$ Let $\imath : C \rightarrow B$ be a smooth inclusion of a circle $C$. Since $\pi_*$ is natural, we obtain
the following commutative diagram:
\begin{equation} \label{diagram: naturality}
\footnotesize \SelectTips{cm}{}  \xymatrix{ H^{p+2}(P)
\ar[r]^-{\pi_*} \ar[d]_-{{\bar{\imath}}^*} & H^{p+1}(B)
\ar[d]^-{\imath^*}     \\
H^{p+2}(P|_C) \ar[r]^-{(\imath^* \pi)_*} & H^{p+1}(C)       }
\end{equation}
where $\imath^* \pi : \imath^* P = P|_C \longrightarrow A$ is the pullback bundle
\begin{equation*} \label{diagram: pullback bundle}
\footnotesize \SelectTips{cm}{} \xymatrix{\imath^* P
\ar[r]^-{\bar{\imath}}
\ar[d]_-{\imath^* \pi} & P \ar[d]^-{\pi} \\
C \ar[r]^-{\imath} & B}
\end{equation*}
Applying $\Hom_\R(\cdot, \R)$ to (\ref{diagram: naturality}) for the case when $p=0,$ we obtain
\begin{equation} \label{diagram: naturality for circle}
\footnotesize \SelectTips{cm}{}  \xymatrix{  H_2 (P)  &
\ar[l]_-{\pi_*^\prime} H_1 (B)       \\
H_2 (P|_C) \ar[u]^-{({\bar{\imath}}^*)^\prime} & \ar[l]_-{(\imath
^*\pi)_*^\prime} H_1 (C) \ar[u]_-{(\imath^*)^\prime} }
\end{equation}

We will apply the Gysin sequence to the principal $S^1$-bundle
$\hat{\pi}|_{\mu^{-1}(s)} : \mu^{-1}(s) \rightarrow B$ for regular
values $s=0, 1.5, 3.5, 5.5, 7.$ For simplicity, denote $\mu^{-1}(s)$ and
$\hat{\pi}|_{\mu^{-1}(s)}$ by $P$ and $\pi$, respectively. To express homology groups of level sets by their
generators, we define some submanifolds of level sets. Pick a circle
$L_i^s$ in $P$ such that
\begin{equation*}
\pi (L_i^s) = L_i
\end{equation*}
for $s=1.5, ~ 3.5, ~ 5.5.$ This is always possible because the
restricted bundle $P|_{L_i}$ is a trivial bundle. Denote by
$L_\fiber^s$ for $s = 1.5, 3.5, 5.5$ a fiber of $\mu^{-1}(s).$ If
exists, pick a torus $L_{ij}^s$ in $P$ such that
\begin{equation*}
\pi (L_{ij}^s) = L_{ij}
\end{equation*}
for $s=1.5, ~ 3.5, ~ 5.5.$ This is possible only when $P|_{L_{ij}}$
is trivial. Denote by $L_{i \fiber}^s$ the bundle $P|_{L_i} = \pi^{-1}(L_i)$ for $s
= 1.5, 3.5, 5.5.$

\subsection{Homology of $\mu^{-1}(s)$ for $s=0$}

In this case, $c_1 (P)=0,$ and $\pi$ is
\begin{equation*}
(x^1, x^2, x^3, x^4, x^5) \mapsto (x^1, x^2, x^3, x^4).
\end{equation*}
It is easy to see that
\begin{equation} \label{homology: n=1,2 s=0}
\begin{array}{lll}
H_1 \big( \mu^{-1}(0) \big)    & = ~ \langle ~ L_i^0 ~|~ 1 \le i \le
4 ~ \rangle + \langle ~ L_\fiber^0 ~ \rangle
& \cong ~ \R^5,  \\
H_2 \big( \mu^{-1}(0) \big)    & = ~ \langle ~ L_{ij}^0 ~|~ 1 \le
i,j \le 4 ~ \rangle + \langle ~ L_{i \fiber}^0 ~|~ 1 \le i \le 4 ~
\rangle    & \cong ~ \R^{10}.
\end{array}
\end{equation}

\subsection{Homology of $\mu^{-1}(s)$ for $s=1.5$}

In this case, $c_1 (P)=-[\sigma_{42}].$ So, $L_{24}^{1.5}$ does not
exist, but other $L_{ij}^{1.5}$'s exist. First, we calculate $H_1
(P).$ Substituting $p=-1$ into (\ref{sequence: gysin}) and
(\ref{sequence: homological gysin}), we obtain the
following:{\footnotesize
\begin{align} \label{sequence: gysin s=1.5 p=-1}
\SelectTips{cm}{} \xymatrix{ H^{-1} (B) = 0 \ar[r]^-{e \wedge} & H^1
(B) \cong \R^4 \ar[r]^-{\pi^*} & H^1 (P) \ar[r]^-{\pi_*} & H^0 (B)
\cong \R \ar[r]^-{e \wedge}  & H^2 (B) \cong \R^6}, \\
\notag \SelectTips{cm}{} \xymatrix{ H_{-1} (B) = 0 & \ar[l]_-{(e
\wedge)^\prime} H_1 (B) \cong \R^4 & \ar[l]_-{(\pi^*)^\prime} H_1
(P) & \ar[l]_-{\pi_*^\prime} H_0 (B) \cong \R & \ar[l]_-{(e
\wedge)^\prime} H_2 (B) \cong \R^6}.
\end{align}
}In (\ref{sequence: gysin s=1.5 p=-1}), the first $e \wedge$ is a
zero map, and the second $e \wedge$ is injective because $e$ is
nontrivial. By exactness,
\begin{equation*}
\ker \pi^* = 0 \quad \text{ and } \quad  \im \pi_* = 0.
\end{equation*}
So,
\begin{equation*}
\im (\pi^*)^\prime = H_1 (B) \quad \text{ and } \quad \ker
\pi_*^\prime = H_0 (B).
\end{equation*}
This means that $(\pi^*)^\prime$ is isomorphic. So,
\begin{equation*}
H_1 \big( \mu^{-1}(1.5) \big) ~ = ~ \langle ~ L_i^{1.5} ~|~ 1 \le i
\le 4 ~ \rangle ~ \cong ~ \R^4.
\end{equation*}
Moreover, any fiber of $P$ is trivial in $H_1 (P)$ because
$(\pi^*)^\prime$ is equal to the homology functor $H_1(\pi)$ and
hence a fiber is sent to 0 by the isomorphism $(\pi^*)^\prime.$

Next, we calculate $H_2 (P).$ Substituting $p=0$ into
(\ref{sequence: gysin}) and (\ref{sequence: homological gysin}), we
obtain the followings:{\footnotesize
\begin{align} \label{sequence: gysin s=1.5 p=0}
\SelectTips{cm}{} \xymatrix{ H^0 (B) \cong \R \ar[r]^-{e \wedge} &
H^2 (B) \cong \R^6 \ar[r]^-{\pi^*} & H^2 (P) \ar[r]^-{\pi_*} & H^1
(B) \cong \R^4 \ar[r]^-{e \wedge}  & H^3 (B) \cong \R^4}, \\
\notag \SelectTips{cm}{} \xymatrix{ H_0 (B) \cong \R & \ar[l]_-{(e
\wedge)^\prime} H_2 (B) \cong \R^6 & \ar[l]_-{(\pi^*)^\prime} H_2
(P) & \ar[l]_-{\pi_*^\prime} H_1 (B) \cong \R^4 & \ar[l]_-{(e
\wedge)^\prime} H_3 (B) \cong \R^4}.
\end{align}
}In (\ref{sequence: gysin s=1.5 p=0}), the image of the first $e
\wedge$ is equal to $\langle \sigma_{42} \rangle,$ and the kernel of
the second $e \wedge$ is equal to $\langle \sigma_2, \sigma_4
\rangle.$ By exactness,
\begin{equation*}
\ker \pi^* = \langle \sigma_{42} \rangle \quad \text{ and } \quad
\im \pi_* = \langle \sigma_2, \sigma_4 \rangle.
\end{equation*}
So,
\begin{align*}
\im (\pi^*)^\prime &= \langle \sigma_{42} \rangle^\perp = \langle
L_{12}, ~ L_{13}, ~ L_{14}, ~ L_{23}, ~ L_{34} \rangle, \\
\ker \pi_*^\prime &= \langle \sigma_2, \sigma_4 \rangle^\perp =
\langle L_1, ~ L_3 \rangle.
\end{align*}
From these, we obtain
\begin{align*}
\im \pi_*^\prime &= \langle ~ \pi_*^\prime (L_2), ~ \pi_*^\prime
(L_4) ~ \rangle, \\
H_2 (P) &= \langle ~ L_{12}^{1.5}, ~ L_{13}^{1.5}, ~ L_{14}^{1.5}, ~
L_{23}^{1.5}, ~ L_{34}^{1.5}, ~ \pi_*^\prime (L_2), ~ \pi_*^\prime
(L_4) ~ \rangle \cong \R^7.
\end{align*}
Substituting $C=L_i$ into (\ref{diagram: naturality for circle}), we
obtain
\begin{equation*}
\pi_*^\prime (L_2) = L_{2 \fiber}^{1.5}, \qquad \pi_*^\prime (L_4) =
L_{4 \fiber}^{1.5} , \qquad L_{1 \fiber}^{1.5} = L_{3 \fiber}^{1.5}
= 0
\end{equation*}
in $H_2 (P).$ That is,
\begin{equation} \label{homology: n=2 s=1.5}
H_2 \big( \mu^{-1}(1.5) \big) = \langle ~ L_{12}^{1.5}, ~
L_{13}^{1.5}, ~ L_{14}^{1.5}, ~ L_{23}^{1.5}, ~ L_{34}^{1.5} ~
\rangle + \langle ~ L_{2 \fiber}^{1.5}, ~ L_{4 \fiber}^{1.5} ~
\rangle   \cong ~ \R^7
\end{equation}
where $L_{1 \fiber}^{1.5} = L_{3 \fiber}^{1.5} = 0$ in $H_2 \big(
\mu^{-1}(1.5) \big).$

\subsection{Homology of $\mu^{-1}(s)$ for $s=3.5$}

In this case, $c_1 (P)=-[\sigma_{31}+\sigma_{42}],$ and
$L_{13}^{3.5}, L_{24}^{3.5}$ do not exist. Also, pick a class in
$H_2 (\pi)^{-1} (L_{13}-L_{24})$ and call it
$(L_{13}-L_{24})^{3.5}.$ This is possible because the pairing of
$c_1 (P)$ and $L_{13}-L_{24}$ is zero.

First, calculation of $H_1 (P)$ is similar to $s=1.5.$ That is,
\begin{equation*}
H_1 \big( \mu^{-1}(3.5) \big) ~ = ~ \langle ~ L_i^{3.5} ~|~ 1 \le i
\le 4 ~ \rangle ~ \cong ~ \R^4,
\end{equation*}
and any fiber of $P$ is trivial in $H_1 (P).$

Next, we calculate $H_2 (P).$ Substituting $p=0$ into
(\ref{sequence: gysin}) and (\ref{sequence: homological gysin}), we
obtain the followings:{\footnotesize
\begin{align} \label{sequence: gysin s=3.5 p=0}
\SelectTips{cm}{} \xymatrix{ H^0 (B) \cong \R \ar[r]^-{e \wedge} &
H^2 (B) \cong \R^6 \ar[r]^-{\pi^*} & H^3 (P) \ar[r]^-{\pi_*} & H^1
(B) \cong \R^4 \ar[r]^-{e \wedge}  & H^3 (B) \cong \R^4}, \\
\notag \SelectTips{cm}{} \xymatrix{ H_0 (B) \cong \R & \ar[l]_-{(e
\wedge)^\prime} H_2 (B) \cong \R^6 & \ar[l]_-{(\pi^*)^\prime} H_2
(P) & \ar[l]_-{\pi_*^\prime} H_1 (B) \cong \R^4 & \ar[l]_-{(e
\wedge)^\prime} H_3 (B) \cong \R^4}.
\end{align}
}In (\ref{sequence: gysin s=3.5 p=0}), the image of the first $e
\wedge$ is equal to $\langle \sigma_{31}+\sigma_{42} \rangle,$ and
the kernel of the second $e \wedge$ is trivial. By exactness,
\begin{equation*}
\ker \pi^* = \langle \sigma_{31}+\sigma_{42} \rangle \quad \text{
and } \quad \im \pi_* = \langle 0 \rangle.
\end{equation*}
So,
\begin{align*}
\im (\pi^*)^\prime &= \langle \sigma_{31}+\sigma_{42} \rangle^\perp
= \langle
L_{12}, ~ L_{13}-L_{24}, ~ L_{14}, ~ L_{23}, ~ L_{34} \rangle, \\
\ker \pi_*^\prime &= \langle 0 \rangle^\perp = H_1 (B).
\end{align*}
From these, we obtain $\im \pi_*^\prime = \langle 0 \rangle,$ and
\begin{equation} \label{homology: n=2 s=3.5}
H_2 \big( \mu^{-1}(3.5) \big) = \langle ~ L_{12}^{3.5}, ~
(L_{13}-L_{24})^{3.5}, ~ L_{14}^{3.5}, ~ L_{23}^{3.5}, ~
L_{34}^{3.5} ~ \rangle \cong \R^5.
\end{equation}
Substituting $C=L_i$ into (\ref{diagram: naturality for circle}), we
obtain $L_{i \fiber}^{3.5} = 0$ in $H_2 (\mu^{-1}(3.5))$ for $1 \le
i \le 4.$

\section{Elementary cobordism} \label{section: elementary cobordism}
In this section, we study an elementary cobordism $\mu^{-1}[a, b]$
for two regular values $a < b$ of $\mu$ such that $\lambda =
\mu(Z^\lambda)$ is the unique critical value between them.

For each critical submanifold $Z^\lambda$ of $X,$ the almost
K\"{a}hler structure $J$ induces a complex structure on the normal
bundle $\nu$ of $Z^\lambda$ in $X$ which splits as a sum $\nu =
\nu^- \oplus \nu^+.$ By the equivariant symplectic neighborhood
theorem, we may assume that a small neighborhood $N_\epsilon$ of $Z$
for a sufficiently small $\epsilon$ is equivariantly symplectically
diffeomorphic to the interior of $D_\epsilon(\nu)$ of the zero
section of $\nu$ where $D_\epsilon(\cdot)$ is the disc bundle with
radius $\epsilon.$ And, we may assume that
$D_\epsilon(\nu_\epsilon^-)$ and $D_\epsilon(\nu_\epsilon^+)$ are
contained in stable and unstable manifolds of $Z^\lambda$ with
respect to the vector field $J \xi,$ respectively. To calculate
homology groups of $\mu^{-1}[a, b],$ we need calculate the first
Chern classes of $\nu^\pm.$ For this, we will show that
\begin{align}
\label{equation: chern class 1} c_1 (\nu^-) &= c_1 \Big(
(\hat{\pi}|_{Z^\lambda})^* \mu^{-1}(a) \Big), \qquad c_1 (\nu^+) =
c_1 \Big(
(\hat{\pi}|_{Z^\lambda})^* \mu^{-1}(b) \Big),   \\
\label{equation: chern class 2} c_1(\nu^-) &= - c_1(\nu^+)
\end{align}
where $\mu^{-1}(a),$ $\mu^{-1}(b)$ are regarded as circle bundles
over $B.$ We can observe that $Z^\lambda$ are connected by $J \xi$
to $\mu^{-1}(a)|_{\hat{\pi}(Z^\lambda)}$ and
$\mu^{-1}(b)|_{\hat{\pi}(Z^\lambda)}.$ Since the flow of $J \xi$ is
equivariant, this observation means that
\begin{equation*}
S_\epsilon (\nu^-) \cong (\hat{\pi}|_{Z^\lambda})^* \mu^{-1}(a)
\quad \text{ and } \quad S_\epsilon (\nu^+) \cong
(\hat{\pi}|_{Z^\lambda})^* \mu^{-1}(b)
\end{equation*}
as $S^1$-bundles where $S_\epsilon (\cdot)$ is the sphere bundle
with radius $\epsilon.$ So, we obtain (\ref{equation: chern class
1}). Also, we obtain (\ref{equation: chern class 2}) because the
normal bundle of $\hat{\pi}(Z^\lambda)$ in $B$ is isomorphic to
$\nu^- \otimes \nu^+$ by \cite[p. 156]{McD} and is trivial. By using
(\ref{equation: chern class 1}) and (\ref{equation: chern class 2}),
we can calculate $c_1(\nu^\pm)$ as follows:
\begin{equation} \label{equation: chern classes}
\begin{array}{ll}
c_1(\nu^-) = -c_1(\nu^+) = 0     \qquad & \text{for } \lambda=1, 6, \\
c_1(\nu^-) = -c_1(\nu^+) = \pm 1 \qquad & \text{for } \lambda=2, 5
\end{array}
\end{equation}
up to orientation of $Z^\lambda.$ When $\hat{\pi}(Z^\lambda) =
L_{ij},$ put $Z_i^\lambda = \hat{\pi}^{-1}(L_i) \cap
Z_{ij}^\lambda.$ When $S_\epsilon (\nu^\pm)$ is trivial, pick a
section of $S_\epsilon (\nu^\pm)$ and denote it by $Z_{ij}^{\lambda,
\pm}.$ And, pick a section of $S_\epsilon (\nu^\pm|_{Z_i^\lambda})$
and denote it by $Z_i^{\lambda, \pm}.$ For simplicity, we also
denote
\begin{equation*}
S_\epsilon (\nu^\pm)|_{Z_i^\lambda}, \quad S_\epsilon
(\nu^\pm)|_{Z_j^\lambda} \qquad \text{by} \qquad Z_{i
\fiber}^{\lambda, \pm}, \quad Z_{j \fiber}^{\lambda, \pm},
\end{equation*}
respectively. And, let $Z_\fiber^{\lambda, \pm}$ be a fiber of
$S_\epsilon (\nu^\pm).$ By using Gysin sequence as in Section
\ref{section: level set}, we can calculate homology groups of sphere
bundles $S_\epsilon (\nu^\pm)$ as follows:
\begin{equation} \label{homology: sphere bundle 1}
H_1(S_\epsilon (\nu^\pm)) = \left\{
  \begin{array}{ll}
    \langle ~ Z_i^{\lambda, \pm}, ~ Z_j^{\lambda, \pm}, ~ Z_\fiber^{\lambda, \pm} ~ \rangle ~
    \cong ~ \R^3
    \qquad  & \text{for } \lambda=1, 6, \\
    \langle ~ Z_i^{\lambda, \pm}, ~ Z_j^{\lambda, \pm} ~ \rangle ~
    \cong ~ \R^2
    \qquad  & \text{for } \lambda=2, 5,
  \end{array}
\right.
\end{equation}
and
\begin{equation} \label{homology: sphere bundle 2}
H_2(S_\epsilon (\nu^\pm)) = \left\{
  \begin{array}{ll}
    \langle Z_{i \fiber}^{\lambda, \pm}, Z_{j \fiber}^{\lambda, \pm},
    Z_{i j}^{\lambda, \pm} \rangle ~ \cong ~ \R^3
    \qquad  & \text{for } \lambda=1, 6, \\
    \langle Z_{i \fiber}^{\lambda, \pm}, Z_{j \fiber}^{\lambda, \pm} \rangle ~
    \cong ~ \R^2
    \qquad  & \text{for } \lambda=2, 5.
  \end{array}
\right.
\end{equation}

\begin{remark} \label{remark: definition of sections}
In this section, $Z_i^{\lambda, \pm}$ and $Z_{ij}^{\lambda, \pm}$
are arbitrary sections by definition. However, in the next section,
we will designate more specific sections to these.
\end{remark}

Before we go further, we review Mayer-Vietoris sequence. For a pair
of subspaces $A, A^\prime$ of a topological space $Y$ such that $Y$
is the union of the interiors of $A, A^\prime,$ this exact sequence
has the form:
\begin{equation} \label{sequence: MV}
\footnotesize \SelectTips{cm}{} \xymatrix{ \cdots ~ \ar[r] & H_{n+1}
(Y) \ar[r]^-{\partial_*} & H_n ( A \cap A^\prime ) \ar[r]^-{(i_*,
j_*)} & H_n (A) \oplus H_n (A^\prime) \ar[r]^-{k_*-l_*}  & H_n (Y) }
\end{equation}
\begin{equation*}
\footnotesize \SelectTips{cm}{} \xymatrix{ \ar[r]^-{\partial_*} &
H_{n-1} ( A \cap A^\prime ) \ar[r] & ~ \cdots ~ \ar[r] & H_0 (A)
\oplus H_0 (A^\prime) \ar[r]^-{k_*-l_*} & H_0 (Y) \ar[r] & 0. }
\end{equation*}
The boundary maps $\partial_*$ lowering the dimension may be made
explicit as follows. An element $y$ in $H_n(Y)$ is the homology
class of an $n$-cycle $Y$ which, by barycentric subdivision for
example, can be written as the sum of two $n$-chains $u$ and $v$
whose images lie wholly in $A$ and $A^\prime,$ respectively. Thus
$\partial y = \partial (u + v) = 0$ so that $\partial u = -\partial
v.$ This implies that the images of both these boundary $(n -
1)$-cycles are contained in the intersection $A \cap A^\prime.$
Then, $\partial_*([x])$ is the class of $\partial u \in H_{n-1} (A
\cap A^\prime).$

Since $\mu$ is a Morse-Bott function, the elementary cobordism
$\mu^{-1} [a, b]$ is homeomorphic to the attaching space
\begin{equation} \label{equation: attach1}
\mu^{-1} [ a, a^\prime ] \quad \bigcup_f \quad D_\epsilon (\nu^-)
\oplus D_\epsilon (\nu^+)
\end{equation}
for $a < a^\prime < \lambda$ with a attaching map $f : S_\epsilon
(\nu^-) \oplus D_\epsilon (\nu^+) \rightarrow \mu^{-1} (a^\prime)$
by \cite[Section 11]{P}, \cite{W}. Here, the restriction of $f$ to
$S_\epsilon (\nu^-)$
\begin{equation*}
f|_{S_\epsilon (\nu^-)} : S_\epsilon (\nu^-) \oplus 0
\longrightarrow \mu^{-1} (a^\prime) \big|_{\hat{\pi}(Z^\lambda)}
\end{equation*}
is defined by the flow of $J \xi.$ Similarly, $\mu^{-1} [a, b]$ is
homeomorphic to the attaching space
\begin{equation} \label{equation: attach2}
D_\epsilon (\nu^-) \oplus D_\epsilon (\nu^+) \quad
\bigcup_{f^\prime} \quad \mu^{-1} [ b^\prime, b ]
\end{equation}
for $\lambda < b^\prime < b$ with a attaching map $f^\prime :
D_\epsilon (\nu^-) \oplus S_\epsilon (\nu^+) \rightarrow \mu^{-1}
(b^\prime)$ whose restriction to $S_\epsilon (\nu^+)$ is defined by
the flow of $J \xi.$ Since $D_\epsilon (\nu^\pm) \oplus S_\epsilon
(\nu^\mp)$ is homotopically equivalent to $S_\epsilon (\nu^\mp),$ we
can obtain $H_1$ and $H_2$ of $D_\epsilon (\nu^\pm) \oplus
S_\epsilon (\nu^\mp)$ by (\ref{homology: sphere bundle 1}),
(\ref{homology: sphere bundle 2}). Furthermore since we already know
homology groups of level sets, we can calculate homology of the
elementary cobordism by using Mayer-Vietoris sequence. More
precisely, we will calculate the first and second homology groups of
two elementary cobordisms
\begin{equation} \label{equation: two elementary cobordisms}
\mu^{-1}[0, 1.5] \quad \text{ and } \quad \mu^{-1}[1.5, 3.5].
\end{equation}
in two ways of (\ref{equation: attach1}), (\ref{equation: attach2}).
Then, we will calculate cohomology of
\begin{equation*}
\mu^{-1}[0, 3.5] ~=~ \mu^{-1}[0, 1.5] ~ \cup ~ \mu^{-1}[1.5, 3.5],
\end{equation*}
and finally $W$ again by Mayer-Vietoris sequence. When we apply
Mayer-Vietoris sequence to attaching (\ref{equation: attach1}) or
(\ref{equation: attach2}) in the below, we will use the following
notations:
\begin{align*}
Y        &= \mu^{-1} [ a, a^\prime ] \quad \bigcup_f \quad
D_\epsilon (\nu^-) \oplus D_\epsilon (\nu^+)
, \\
A        &= \mu^{-1} [ a, a^\prime ], \qquad A^\prime = D_\epsilon
(\nu^-) \oplus D_\epsilon (\nu^+),
\end{align*}
or
\begin{align*}
Y        &= D_\epsilon (\nu^-) \oplus D_\epsilon (\nu^+) \quad
\bigcup_{f^\prime} \quad \mu^{-1} [ b^\prime, b ]
, \\
A        &= D_\epsilon (\nu^-) \oplus D_\epsilon (\nu^+), \qquad
A^\prime = \mu^{-1} [ b^\prime, b ],
\end{align*}
respectively.


\section{The first homology groups of elementary cobordisms}
\label{section: H_1 of elementary cobordism}

In this section, we calculate the first homology group of elementary
cobordisms of (\ref{equation: two elementary cobordisms}). As stated
in Remark \ref{remark: definition of sections}, we will define
$Z_i^{\lambda, \pm},$ $Z_{i \fiber}^{\lambda, \pm},$
$Z_{ij}^{\lambda, \pm}$ more precisely by specifying their images
through $i_*$ in Mayer-Vietoris sequences appearing in this and the
next sections.

Let $\mu^{-1} [ a, b ]$ be one of (\ref{equation: two elementary
cobordisms}), and let $\lambda$ be the unique critical value between
$a$ and $b.$ Let $\nu^\pm$ be normal bundles over $Z^\lambda.$ To
calculate $H_1(\mu^{-1} [ a, b ]),$ we apply Mayer-Vietoris sequence
to (\ref{equation: attach1}), (\ref{equation: attach2}) for $n=1,$
$\lambda=1, 2$ as follows:
\begin{enumerate}
  \item when we apply Mayer-Vietoris sequence to (\ref{equation: attach1}) for
$\lambda=1,$ { \footnotesize
\begin{equation} \label{sequence: MV n=1 l=1 attach1}
\begin{array}{ccccc}
H_1 ( A \cap A^\prime ) \cong \R^3 & ~^{\underrightarrow{~ (i_*,
j_*) ~}} & H_1 (A) \oplus H_1 (A^\prime) \cong \R^5 \oplus \R^2 &
~^{\underrightarrow{~ k_*-l_* ~}} & H_1 (Y) \rightarrow 0 \\
Z_1^{1, -}        & \longmapsto & (L_1^0, Z_1^1)     & & \\
Z_3^{1, -}        & \longmapsto & (L_3^0, Z_3^1)     & & \\
Z_\fiber^{1, -}   & \longmapsto & (L_\fiber^0, 0),   & &
\end{array}
\end{equation}
}
  \item when we apply Mayer-Vietoris sequence to (\ref{equation: attach2}) for
$\lambda=1,$ { \footnotesize
\begin{equation} \label{sequence: MV n=1 l=1 attach2}
\begin{array}{ccccc}
H_1 ( A \cap A^\prime ) \cong \R^3 & ~^{\underrightarrow{~ (i_*,
j_*) ~}} & H_1 (A) \oplus H_1 (A^\prime) \cong \R^2 \oplus \R^4 &
~^{\underrightarrow{~ k_*-l_* ~}} & H_1 (Y) \rightarrow 0 \\
Z_1^{1, +}             & \longmapsto & (Z_1^1, L_1^{1.5}) & & \\
Z_3^{1, +}             & \longmapsto & (Z_3^1, L_3^{1.5}) & & \\
Z_\fiber^{\lambda, +}  & \longmapsto & (0, 0),            & &
\end{array}
\end{equation}
}
  \item when we apply Mayer-Vietoris sequence to (\ref{equation: attach1}) for
$\lambda=2,$ { \footnotesize
\begin{equation} \label{sequence: MV n=1 l=2 attach1}
\begin{array}{ccccc}
H_1 ( A \cap A^\prime ) \cong \R^2 & ~^{\underrightarrow{~ (i_*,
j_*) ~}} & H_1 (A) \oplus H_1 (A^\prime) \cong \R^4 \oplus \R^2 &
~^{\underrightarrow{~ k_*-l_* ~}} & H_1 (Y) \rightarrow 0 \\
Z_2^{2, -} & \longmapsto & (L_2^{1.5}, Z_2^2)    & & \\
Z_4^{2, -} & \longmapsto & (L_4^{1.5}, Z_4^2),   & &
\end{array}
\end{equation}
}
  \item when we apply Mayer-Vietoris sequence to (\ref{equation: attach2}) for
$\lambda=2,$ { \footnotesize
\begin{equation} \label{sequence: MV n=1 l=2 attach2}
\begin{array}{ccccc}
H_1 ( A \cap A^\prime ) \cong \R^2 & ~^{\underrightarrow{~ (i_*,
j_*) ~}} & H_1 (A) \oplus H_1 (A^\prime) \cong \R^2 \oplus \R^4 &
~^{\underrightarrow{~ k_*-l_* ~}} & H_1 (Y) \rightarrow 0 \\
Z_2^{2, +} & \longmapsto & (Z_2^2, L_2^{3.5})   & & \\
Z_4^{2, +} & \longmapsto & (Z_4^2, L_4^{3.5}).  & &
\end{array}
\end{equation}
}
\end{enumerate}
By observing these sequences, we can show the followings:
\begin{itemize}
  \item[i.] $H_1(\mu^{-1} [a, b]) \cong \R^4,$
  \item[ii.] $H_1 (\hat{\pi}) : H_1 \big( \mu^{-1} [a, b] \big)
  \longrightarrow H_1 (B)$ is isomorphic,
  \item[iii.] For $s=a$ or $b,$ inclusion induces { \footnotesize
\begin{equation} \label{inclusion: n=1}
\begin{array}{rcc}
H_1 ( \mu^{-1}(s) )  & \longrightarrow & H_1 (\mu^{-1}[a, b])   \\
        L_\fiber^s   & \longmapsto & 0,         \\
             L_i^s   & \longmapsto & L_i^s
\end{array}
\end{equation}
}where $L_\fiber^s$ might be trivial in $H_1 ( \mu^{-1}(s)).$
\end{itemize}

\section{The second homology groups of elementary cobordisms}
\label{section: H_2 of elementary cobordism}

In this section, we calculate the second homology groups of
elementary cobordisms of (\ref{equation: two elementary
cobordisms}). For this, we apply Mayer-Vietoris sequence to
(\ref{equation: attach1}), (\ref{equation: attach2}) for $n=2,$
$\lambda=1, 2.$ As we have showed three thing on the first homology
in the previous section, we will describe the second homology groups
by their generators, and deal with maps induced by $\hat{\mu}$ and
inclusions.

\subsection{$H_2$ of $\mu^{-1}[0, 1.5]$}

Applying Mayer-Vietoris sequence to (\ref{equation: attach1}) for
$n=2,$ $\lambda=1,$ we obtain the followings: { \footnotesize
\begin{equation} \label{sequence: MV n=2 l=1 attach1}
\begin{array}{rclcc}
H_2 ( A \cap A^\prime ) \cong \R^3 & ~^{\underrightarrow{~ (i_*,
j_*) ~}} & H_2 (A) \oplus H_2 (A^\prime) \cong \R^{10} \oplus \R^1 &
~^{\underrightarrow{~ k_*-l_* ~}} & H_2 (Y) ~^{\underrightarrow{~ \partial_* ~}} \\
Z_{1\fiber}^{1, -}      & \longmapsto & (L_{1 \fiber}^0, 0),        & & \\
Z_{3 \fiber}^{1, -}     & \longmapsto & (L_{3 \fiber}^0, 0),        & & \\
Z_{13}^{1, -}           & \longmapsto & (L_{13}^0, Z_{13}^1). & & \\
\end{array}
\end{equation}
}So, the rank of $\im (i_*, j_*)$ is rank 3. Since $(i_*, j_*)$ for
$n=1$ is injective, $\partial_*$ is a zero-map. This implies that
$H_1 ( Y )$ is rank 3, and we can check that
\begin{equation} \label{homology: n=2 l=1 attach1}
H_2 (\mu^{-1}[0, 1.5]) = \big\langle L_{12}^0, L_{13}^0, L_{14}^0,
L_{23}^0, L_{24}^0, L_{34}^0 \big\rangle + \big\langle L_{2
\fiber}^0, L_{4 \fiber}^0 \big\rangle \cong \R^8
\end{equation}
in which $L_{13}^0 = Z_{13}^1$ because $(L_{13}^0, Z_{13}^1)$ is
contained in $\im (i_*, j_*).$ The map $\hat{\mu}$ induces the
following map: { \footnotesize
\begin{equation} \label{project: n=2 l=1 attach1}
\begin{array}{rcl}
H_2 ( \mu^{-1}[0, 1.5] ) \cong \R^8 & \longrightarrow & H_2 (T^4) \cong \R^6  \\
L_{i \fiber}^0      & \longmapsto & 0      \quad \qquad \text{ for } i=2, 4,         \\
L_{ij}^0            & \longmapsto & L_{ij} \qquad \text{ for } 1 \le i \ne j \le 4.  \\
\end{array}
\end{equation}
}And, the inclusion induces the following surjection: {
\footnotesize
\begin{equation} \label{inclusion: n=2 l=1 attach1}
\begin{array}{rcl}
H_2 ( \mu^{-1}(0) ) \cong \R^{10}  & \longrightarrow & H_2 (\mu^{-1}[0, 1.5]) \cong \R^8  \\
L_{i \fiber}^0    & \longmapsto & 0 \qquad \quad       \text{ for } i=1, 3,         \\
L_{i \fiber}^0    & \longmapsto & L_{i \fiber}^0 \qquad \text{ for } i=2, 4,     \\
L_{ij}^0          & \longmapsto & L_{ij}^0 \qquad       \text{ for } 1 \le i \ne j \le 4. \\
\end{array}
\end{equation}
}

Applying Mayer-Vietoris sequence to (\ref{equation: attach2}) for
$n=2,$ $\lambda=1,$ we obtain the followings: { \footnotesize
\begin{equation} \label{sequence: MV n=2 l=1 attach2}
\begin{array}{rclcc}
H_2 ( A \cap A^\prime ) \cong \R^3 & ~^{\underrightarrow{~ (i_*,
j_*) ~}} & H_2 (A) \oplus H_2 (A^\prime) \cong \R^1 \oplus \R^7 &
~^{\underrightarrow{~ k_*-l_* ~}} & H_2 (Y) ~^{\underrightarrow{~ \partial_* ~}} \\
Z_{1 \fiber}^{1, +}, ~ Z_{3 \fiber}^{1, +}    & \longmapsto & 0,        & & \\
Z_{13}^{1, +}                                 & \longmapsto & (Z_{13}^1, L_{13}^{1.5}). & & \\
\end{array}
\end{equation}
}By this,
\begin{equation*}
H_2 (\mu^{-1}[0, 1.5]) \supset \big\langle L_{12}^{1.5},
L_{13}^{1.5}, L_{14}^{1.5}, L_{23}^{1.5}, L_{34}^{1.5} \big\rangle +
\big\langle L_{2 \fiber}^{1.5}, L_{4 \fiber}^{1.5} \big\rangle \cong
\R^7,
\end{equation*}
and $Z_{13}^1=L_{13}^{1.5}.$ Comparing this with (\ref{sequence: MV
n=2 l=1 attach1}), we have
\begin{equation} \label{homology: n=2 l=1 attach2}
H_2 (\mu^{-1}[0, 1.5]) = \big\langle L_{12}^{1.5}, L_{13}^{1.5},
L_{14}^{1.5}, L_{23}^{1.5}, L_{34}^{1.5} \big\rangle + \big\langle
L_{2 \fiber}^{1.5}, L_{4 \fiber}^{1.5} \big\rangle + \big\langle
L_{24}^0 \big\rangle \cong \R^8.
\end{equation}
The map $\hat{\mu}$ induces the following surjection: {
\footnotesize
\begin{equation} \label{project: n=2 l=1 attach2}
\begin{array}{rcl}
H_2 ( \mu^{-1}[0, 1.5] ) \cong \R^8  & \longrightarrow  & H_2 (T^4) \cong \R^6   \\
L_{i \fiber}^{1.5}                   & \longmapsto      & 0 \quad \qquad \text{ for } i=2, 4,         \\
L_{ij}^{1.5}                         & \longmapsto      & L_{ij}
\qquad \text{ for } 1 \le i \ne j \le 4, ~ \{i,j\} \ne \{2,4\}, \\
L_{24}^0                             & \longmapsto      & L_{24}.  \\
\end{array}
\end{equation}
}And, the inclusion induces the following injection: { \footnotesize
\begin{equation} \label{inclusion: n=2 l=1 attach2}
\begin{array}{rcl}
H_2 ( \mu^{-1}(1.5) ) \cong \R^7 & \longrightarrow & H_2 (\mu^{-1}[0, 1.5]) \cong \R^8   \\
L_{i \fiber}^{1.5}   & \longmapsto & L_{i \fiber}^{1.5}  \qquad \text{ for } i=2, 4,     \\
L_{ij}^{1.5}         & \longmapsto & L_{ij}^{1.5} \qquad \text{ for
} 1 \le i \ne j \le 4, ~ \{i,j\} \ne \{2,4\},
\end{array}
\end{equation}
}and $L_{24}^0$ is not contained in its image.

Comparing (\ref{project: n=2 l=1 attach1}) with (\ref{project: n=2
l=1 attach2}), maps (\ref{inclusion: n=2 l=1 attach1}),
(\ref{inclusion: n=2 l=1 attach2}) give us relations between
$Z_{13}^1$ and all generators of $\mu^{-1}(0),$ $\mu^{-1}(1.5)$ as
follows: { \footnotesize
\begin{equation}
\label{relation: n=2 l=1}
\begin{array}{ll}
L_{ij}^0 = L_{ij}^{1.5}             & \qquad \text{for } 1 \le i \ne
j \le
4, ~\{i,j\} \ne \{2,4\}, \\
L_{i\fiber}^0 = L_{i\fiber}^{1.5}   & \qquad \text{for } i=2, 4, \\
L_{i\fiber}^0 = 0                   & \qquad \text{for } i=1, 3, \\
Z_{13}^1 = L_{13}^0 = L_{13}^{1.5}, & \\
L_{24}^0 = L_{24}^0                 &
\end{array}
\end{equation}
}up to $\langle L_{2\fiber}^0, L_{4\fiber}^0 \rangle = \langle
L_{2\fiber}^{1.5}, L_{4\fiber}^{1.5} \rangle$ in $H_2 ( \mu^{-1}[0,
1.5] )$ where $L_{24}^0 = L_{24}^0$ means that there exists no
relation on $L_{24}^0.$

\subsection{$H_2$ of $\mu^{-1}[1.5, 3.5]$}

Applying Mayer-Vietoris sequence to (\ref{equation: attach1}) for
$n=2,$ $\lambda=2,$ we obtain the followings: { \footnotesize
\begin{equation} \label{sequence: MV n=2 l=2 attach1}
\begin{array}{rcccc}
H_2 ( A \cap A^\prime ) \cong \R^2 & ~^{\underrightarrow{~ (i_*,
j_*) ~}} & H_2 (A) \oplus H_2 (A^\prime) \cong \R^7 \oplus \R^1 &
~^{\underrightarrow{~ k_*-l_* ~}} & H_2 (Y) ~^{\underrightarrow{~ \partial_* ~}} \\
Z_{2 \fiber}^{2, -}  & \longmapsto & (L_{2 \fiber}^{1.5}, 0),        & & \\
Z_{4 \fiber}^{2, -}  & \longmapsto & (L_{4 \fiber}^{1.5}, 0).
& &
\end{array}
\end{equation}
}So, the rank of $\im (i_*, j_*)$ is rank two. Since $(i_*, j_*)$
for $n=1$ is injective, $\partial_*$ is a zero-map. This implies
that $H_1 ( Y )$ is rank six, and we can check that
\begin{equation} \label{homology: n=2 l=2 attach1}
H_2 (\mu^{-1}[1.5, 3.5]) = \big\langle L_{12}^{1.5}, L_{13}^{1.5},
L_{14}^{1.5}, L_{23}^{1.5}, L_{34}^{1.5} \big\rangle \oplus
\big\langle Z_{42}^2 \rangle \cong \R^6.
\end{equation}
The map $\hat{\mu}$ induces the following isomorphism: {
\footnotesize
\begin{equation} \label{project: n=2 l=2 attach1}
\begin{array}{rcl}
H_2 ( \mu^{-1}[1.5, 3.5] ) \cong \R^6 & \longrightarrow & H_2 (T^4) \cong \R^6   \\
Z_{42}^2   & \longmapsto & L_{42},         \\
L_{ij}^{1.5}   & \longmapsto & L_{ij} \qquad \text{ for } 1 \le i
\ne j \le 4, ~ \{i,j\} \ne \{2,4\}.
\end{array}
\end{equation}
}And, inclusion induces the following map: { \footnotesize
\begin{equation} \label{inclusion: n=2 l=2 attach1}
\begin{array}{rcl}
H_2 ( \mu^{-1}(1.5) ) \cong \R^7            & \longrightarrow & H_2 (\mu^{-1}[1.5, 3.5]) \cong \R^6   \\
L_{i \fiber}^{1.5}             & \longmapsto & 0 ~ \quad \qquad \text{ for } i=2, 4,         \\
L_{ij}^{1.5}                   & \longmapsto & L_{ij}^{1.5} \qquad
\text{ for } 1 \le i \ne j \le 4, ~ \{i,j\} \ne \{2,4\},
\end{array}
\end{equation}
}and $Z_{42}^2$ is not contained in its image.

Applying Mayer-Vietoris sequence to (\ref{equation: attach2}) for
$n=2,$ $\lambda=2,$ we obtain the followings: { \footnotesize
\begin{equation} \label{sequence: MV n=2 l=2 attach2}
\begin{array}{ccccc}
H_2 ( A \cap A^\prime ) \cong \R^2 & ~^{\underrightarrow{~ (i_*,
j_*) ~}} & H_2 (A) \oplus H_2 (A^\prime) \cong \R^1 \oplus \R^5 &
~^{\underrightarrow{~ k_*-l_* ~}} & H_2 (Y) ~^{\underrightarrow{~ \partial_* ~}} \\
Z_{2 \fiber}^{2, +}, ~ Z_{4 \fiber}^{2, +}  & \longmapsto & 0.        & & \\
\end{array}
\end{equation}
}So, the rank of $\im (i_*, j_*)$ is rank zero. Since $(i_*, j_*)$
for $n=1$ is injective, $\partial_*$ is a zero-map. This implies
that $H_1 ( Y )$ is rank six, and we can check that
\begin{equation} \label{homology: n=2 l=2 attach2}
H_2 (\mu^{-1}[1.5, 3.5]) = \big\langle L_{12}^{3.5},
(L_{13}-L_{24})^{3.5}, L_{14}^{3.5}, L_{23}^{3.5}, L_{34}^{3.5}
\rangle \oplus \big\langle Z_{42}^2 \big\rangle \cong \R^6.
\end{equation}
The map $\hat{\mu}$ induces the following isomorphism: {
\footnotesize
\begin{equation} \label{project: n=2 l=2 attach2}
\begin{array}{rcl}
H_2 ( \mu^{-1}[1.5, 3.5] ) \cong \R^6 & \longrightarrow & H_2 (T^4) \cong \R^6   \\
(L_{13}-L_{24})^{3.5}   & \longmapsto & L_{13} - L_{24},         \\
Z_{42}^2       & \longmapsto & L_{42},  \\
L_{ij}^{3.5}   & \longmapsto & L_{ij} \qquad \text{ for } 1 \le i
\ne j \le 4, ~ \{i,j\} \ne \{1,3\}, \{2,4\}.
\end{array}
\end{equation}
}And, the inclusion induces the following injection: { \footnotesize
\begin{equation} \label{inclusion: n=2 l=2 attach2}
\begin{array}{rcl}
H_2 ( \mu^{-1}(3.5) ) \cong \R^5 & \longrightarrow & H_2 (\mu^{-1}[1.5, 3.5]) \cong \R^6   \\
(L_{13}-L_{24})^{3.5}   & \longmapsto & (L_{13}-L_{24})^{3.5},         \\
L_{ij}^{3.5}   & \longmapsto & L_{ij}^{3.5} \qquad \text{ for } 1
\le i \ne j \le 4, ~ \{i,j\} \ne \{1,3\}, \{2,4\},
\end{array}
\end{equation}
}and $Z_{42}^2$ is not contained in its image.

Comparing (\ref{project: n=2 l=2 attach1}) with (\ref{project: n=2
l=2 attach2}), maps (\ref{inclusion: n=2 l=2 attach1}),
(\ref{inclusion: n=2 l=2 attach2}) give us relations between
$Z_{24}^2$ and all generators of $\mu^{-1}(1.5),$ $\mu^{-1}(3.5)$ as
follows: { \footnotesize
\begin{equation} \label{relation: n=2 l=2}
\begin{array}{ll}
L_{13}^{1.5}-Z_{24}^2 = (L_{13}-L_{24})^{3.5},  &   \\
L_{ij}^{1.5}          = L_{ij}^{3.5}            &
\qquad \text{for } 1 \le i \ne j \le 4, ~\{i,j\} \ne \{1,3\}, \{2,4\}, \\
L_{i\fiber}^{1.5}     = 0                       & \qquad \text{for } i=2, 4, \\
\end{array}
\end{equation}
}in $H_2 ( \mu^{-1}[1.5, 3.5] ).$

\section{Homology groups of the union of two elementary cobordisms}
\label{section: homology of [0, 3.5]}

In this section, we calculate homology groups of $\mu^{-1} [0,
3.5].$ Put
\begin{equation} \label{equation: union of two elementary
cobordisms}
Y = \mu^{-1} [0, 3.5], \qquad  A = \mu^{-1} [0, 1.5], \qquad
A^\prime = \mu^{-1} [1.5, 3.5].
\end{equation}
Applying Mayer-Vietoris sequence to (\ref{equation: union of two
elementary cobordisms}) for $n=1,$ we obtain the followings:

{ \footnotesize
\begin{equation} \label{sequence: MV n=1 l=1,2}
\begin{array}{cccccc}
H_1 ( A \cap A^\prime ) \cong \R^4 & ~^{\underrightarrow{~ (i_*,
j_*) ~}} & H_1 (A) \oplus H_1 (A^\prime) \cong \R^4 \oplus \R^4 &
~^{\underrightarrow{~ k_*-l_* ~}} & H_1 (Y) &~^{\underrightarrow{~
\partial_* ~}}  \\
H_0 ( A \cap A^\prime ) \cong \R^1 & ~^{\underrightarrow{~ (i_*,
j_*) ~}} & H_0 (A) \oplus H_0 (A^\prime) \cong \R^1 \oplus \R^1. &
 &   &  \\
\end{array}
\end{equation}
}Since two $(i_*, j_*)$'s are injective, $H_1 (\mu^{-1} [0, 3.5])$
is rank 4, and
\begin{equation*}
H_1 (\hat{\pi}) : H_1 \big( \mu^{-1} [0, 3.5] \big) \longrightarrow
H_1 (B)
\end{equation*}
is isomorphic.

Applying Mayer-Vietoris sequence to (\ref{equation: union of two
elementary cobordisms}) for $n=2,$ we obtain the followings: {
\footnotesize
\begin{equation} \label{sequence: MV n=2 l=1,2}
\begin{array}{rcccc}
H_2 ( A \cap A^\prime ) \cong \R^7 & ~^{\underrightarrow{~ (i_*,
j_*) ~}} & H_2 (A) \oplus H_2 (A^\prime) \cong \R^8 \oplus \R^6 &
~^{\underrightarrow{~ k_*-l_* ~}}
& H_2 (Y) ~^{\underrightarrow{~ \partial_* ~}} \\
L_{2 \fiber}^{1.5}  & \longmapsto  & (L_{2 \fiber}^{1.5}, 0),        & & \\
L_{4 \fiber}^{1.5}  & \longmapsto  & (L_{4 \fiber}^{1.5}, 0),        & & \\
L_{ij}^{1.5}        & \longmapsto  & (L_{ij}^{1.5}, L_{ij}^{1.5}).       & & \\
\end{array}
\end{equation}
}Since $(i_*, j_*)$ for $n=1$ is injective as we have seen in
(\ref{sequence: MV n=1 l=1,2}), $\partial_*$ is zero-map and hence
$k_*-l_*$ is surjective by exactness. Also since $(i_*, j_*)$ for
$n=2$ is injective, the rank of $\im k_*-l_*$ is 7. So, $H_2
(\mu^{-1} [0, 3.5])$ is rank 7. Next, we find generators of $H_2
(\mu^{-1} [0, 3.5]).$ By exactness, $\im (i_*, j_*)$ for $n=2$ gives
relation in $H_2 (Y).$ More precisely, two elements
\begin{equation*}
(L_{2 \fiber}^{1.5}, 0), ~ (L_{4 \fiber}^{1.5}, 0) \in H_2 (A)
\oplus H_2 (A^\prime)
\end{equation*}
give the relation $L_{2 \fiber}^{1.5}=0,$ $L_{4 \fiber}^{1.5}=0$ in
$H_2 (Y),$ respectively. By using (\ref{relation: n=2 l=1}),
(\ref{relation: n=2 l=2}) in addition to these relations, we can
check that $L_{12}^{1.5}, L_{13}^{1.5}, L_{14}^{1.5}, L_{23}^{1.5},
L_{34}^{1.5}, L_{24}^0, Z_{24}^2$ generate $H_2 (\mu^{-1} [0,
3.5]).$ And, this implies { \footnotesize
\begin{equation} \label{homology: n=2 l=1,2}
H_2 (\mu^{-1}[0, 3.5]) = \big\langle L_{12}^{1.5}, L_{13}^{1.5},
L_{14}^{1.5}, L_{23}^{1.5}, L_{34}^{1.5} \big\rangle + \big\langle
L_{24}^0, Z_{24}^2 \big\rangle \cong \R^7
\end{equation}
}because $H_2 (\mu^{-1} [0, 3.5])$ is rank 7. In this way, we can
also show that relations between $Z_{13}^1,$ $Z_{24}^2,$ and all
generators of $\mu^{-1}(0),$ $\mu^{-1}(1.5),$ $\mu^{-1}(3.5)$ are as
follows: { \footnotesize
\begin{equation} \label{relation: n=2 l=1,2}
\begin{array}{ll}
L_{ij}^0 = L_{ij}^{1.5} = L_{ij}^{3.5}                     &
\qquad \text{for } 1 \le i \ne j \le 4, ~\{i,j\} \ne \{1,3\}, \{2,4\}, \\
L_{i\fiber}^0 = L_{i\fiber}^{1.5} = L_{i\fiber}^{3.5} = 0  & \qquad \text{for } i=1,2,3,4, \\
L_{13}^0 = L_{13}^{1.5} = Z_{13}^1,                        &                            \\
L_{24}^0 = L_{24}^0,                                       &                            \\
(L_{13}-L_{24})^{3.5} = L_{13}^{1.5} - Z_{24}^2            &                            \\
\end{array}
\end{equation}
}in $H_2 ( \mu^{-1}[0,3.5] ).$

%

In the exactly same way with $H_n(\mu^{-1} [0, 3.5])$ for $n=1,2,$
we can calculate $H_n(\mu^{-1} [3.5, 7])$ for $n=1,2$ as follows:  {
\footnotesize
\begin{equation} \label{homology: n=2 l=5,6}
H_2 (\mu^{-1}[3.5, 7]) = \big\langle L_{12}^{5.5}, L_{14}^{5.5},
L_{23}^{5.5}, L_{24}^{5.5}, L_{34}^{5.5} \big\rangle + \big\langle
L_{13}^7, Z_{13}^5 \big\rangle \cong \R^7.
\end{equation}
}And we can also show that relations between $Z_{13}^5,$ $Z_{24}^6,$
and all generators of $\mu^{-1}(3.5),$ $\mu^{-1}(5.5),$
$\mu^{-1}(7)$ are as follows:

{ \footnotesize
\begin{equation} \label{relation: n=2 l=5,6}
\begin{array}{ll}
L_{ij}^{3.5} = L_{ij}^{5.5} = L_{ij}^7                     &
\qquad \text{for } 1 \le i \ne j \le 4, ~\{i,j\} \ne \{1,3\}, \{2,4\}, \\
L_{i\fiber}^{3.5} = L_{i\fiber}^{5.5} = L_{i\fiber}^7 = 0  & \qquad \text{for } i=1,2,3,4,  \\
L_{24}^{5.5} = L_{24}^7 = Z_{24}^6,                        &                            \\
L_{13}^7 = L_{13}^7,                                       &                            \\
(L_{13}-L_{24})^{3.5} = L_{42}^{5.5} - Z_{31}^5            &                            \\
\end{array}
\end{equation}
}in $H_2 ( \mu^{-1}[3.5, 7] ).$

\begin{remark}
We can observe that if we exchange $i,j,s,\lambda$ of generators of
the homology group $H_n ( \mu^{-1}[0, 3.5] )$ for $n=1,2$ through
\begin{equation*}
i \mapsto 5-i, \qquad j \mapsto 5-j, \qquad s \mapsto 7-s, \qquad
\lambda \mapsto 7-\lambda,
\end{equation*}
for example
\begin{equation*}
L_2^{1.5} \mapsto L_3^{5.5}, \qquad L_{13}^0 \mapsto L_{42}^7,
\qquad Z_{13}^1 \mapsto Z_{42}^6,
\end{equation*}
then we obtain $H_n ( \mu^{-1}[3.5, 7] )$ for $n=1,2.$ This is also
applicable to their relations. The reason for this relation between
$H_n ( \mu^{-1}[0, 3.5] )$ and $H_n ( \mu^{-1}[3.5, 7] )$ can be
found in \cite[p. 157]{McD}.
\end{remark}

\section{Homology groups of $W$} \label{section: homology of W}

In this section, we calculate homology groups of $W.$ The manifold
$W$ can be considered as the union of $\mu^{-1}[0, 3.5]$ and
$\mu^{-1}[3.5, 7].$ Put
\begin{equation*}
Y = W, \qquad  A = \mu^{-1} [0, 3.5], \qquad A^\prime = \mu^{-1}
[3.5, 7].
\end{equation*}
Then, $A \cap A^\prime = \mu^{-1}(0) \cup \mu^{-1}(3.5).$
And,
\begin{align*}
H_2(A \cap A^\prime) ~ &= ~ H_2(\mu^{-1}(0)) \oplus
H_2(\mu^{-1}(3.5)) ~ \cong ~ \R^{10} \oplus \R^5,  \\
H_1(A \cap A^\prime) ~ &= ~ H_1(\mu^{-1}(0)) \oplus
H_1(\mu^{-1}(3.5)) ~ \cong ~ \R^5 \oplus \R^4,  \\
H_0(A \cap A^\prime) ~ &= ~ H_0(\mu^{-1}(0)) \oplus
H_0(\mu^{-1}(3.5)) ~ \cong ~ \R^1 \oplus \R^1.
\end{align*}
Henceforward, we fix orders of these summands. Applying
Mayer-Vietoris sequence to this for $n=0, 1,$ we obtain the
followings:

{ \footnotesize
\begin{equation} \label{sequence: MV n=1 l=1,2,5,6}
\begin{array}{ccccc}
H_1 ( A \cap A^\prime ) \cong \R^5 \oplus \R^4  &
~^{\underrightarrow{~ (i_*, j_*) ~}}  & H_1 (A) \oplus H_1 (A^\prime)
\cong \R^4 \oplus \R^4  & ~^{\underrightarrow{~ k_*-l_* ~}}  & H_1 (Y) \\
(0, L_i^{3.5})        & \longmapsto  & (L_i^{3.5}, L_i^{3.5})  & &  \\
(L_\fiber^0, 0)       & \longmapsto  & (0, 0)                  & &  \\
(L_j^0, 0), ~ j=1,2   & \longmapsto  & (L_j^0, L_{j+2}^7)      & &  \\
(L_j^0, 0), ~ j=3,4   & \longmapsto  & (L_j^0, L_{j-2}^7)      & &  \\
                      &              &                         & &  \\
~^{\underrightarrow{~ \partial_* ~}} \quad H_0 ( A \cap A^\prime ) \cong
\R^2  & ~^{\underrightarrow{~ (i_*, j_*) ~}}  & H_0 (A) \oplus H_0
(A^\prime) \cong \R^1 \oplus \R^1.  &  &     \\
y_0                   & \longmapsto  & (y_0, y_0)              & &  \\
y_1                   & \longmapsto  & (y_1, y_1).             & &  \\
\end{array}
\end{equation}
}Here, we use the fact that the level set $\mu^{-1}(0)$ is glued to
$\mu^{-1}(7)$ by the involution $\tau.$ The image and kernel of
$(i_*, j_*)$ for $n=1$ are as follows: {\footnotesize
\begin{align}
\label{equation: im of i, j for l=1,2,5,6} \im (i_*, j_*) ~ &= ~
\big\langle (L_i^{3.5}, L_i^{3.5}) ~|~ 1 \le i \le 4 \big\rangle +
\big\langle (L_1^0, L_3^7), (L_2^0, L_4^7) \big\rangle
~ \cong ~ \R^6,    \\
\label{equation: ker of i, j for l=1,2,5,6} \ker (i_*, j_*) ~ &= ~
\Big\langle (L_1^0+L_3^0, -L_1^{3.5}-L_3^{3.5}), ~ (L_2^0+L_4^0,
-L_2^{3.5}-L_4^{3.5}), ~ (L_\fiber^0, 0) \Big\rangle ~ \cong ~ \R^3.
\end{align}
}And, the image and kernel of $(i_*, j_*)$ for $n=0$ are as follows:
\begin{align*}
\im (i_*, j_*)   &=  \big\langle (y_0, y_0) = (y_1, y_1) \big\rangle
\cong \R^1,    \\
\ker (i_*, j_*)  &=  \big\langle (-y_0, y_1) \big\rangle \cong \R^1.
\end{align*}
We can show that $\partial_*(\gamma) = (-y_0, y_1)$ by definition of
$\partial_*.$ By these, we can conclude
\begin{equation} \label{homology: n=1 l=1,2,5,6}
H_1(W) = \big\langle L_1^0, L_2^0, \gamma \big\rangle \cong \R^3
\end{equation}
where $L_1^0=L_3^0, L_2^0=L_4^0.$ This is the proof of Proposition
\ref{proposition: H_1}.

Applying Mayer-Vietoris sequence to this for $n=2,$ we obtain the
followings: { \footnotesize
\begin{equation} \label{sequence: MV n=1 l=1,2,5,6}
\begin{array}{rclcl}
H_2 ( A \cap A^\prime )     & ^{\underrightarrow{(i_*, j_*)}} & H_2 (A) \oplus H_2 (A^\prime)   & ^{\underrightarrow{k_*-l_*}} & H_2 (Y) \qquad \qquad ^{\underrightarrow{~ \partial_* ~}} \\
\cong \R^{10} \oplus \R^5   &             & \cong \R^7 \oplus \R^7                              &                 &                                                                \\
(L_{i \fiber}^0, 0)         & \longmapsto & (0,0)                                               &                 &                                                                \\
(L_{12}^0, 0)               & \longmapsto & (L_{12}^0, L_{34}^7) = (L_{12}^{1.5}, L_{34}^{5.5}) & \dashrightarrow & L_{12}^{1.5} = L_{34}^{5.5}                                    \\
(L_{14}^0, 0)               & \longmapsto & (L_{14}^0, L_{32}^7) = (L_{14}^{1.5}, L_{32}^{5.5}) & \dashrightarrow & L_{14}^{1.5} = L_{32}^{5.5}                                    \\
(L_{23}^0, 0)               & \longmapsto & (L_{23}^0, L_{41}^7) = (L_{23}^{1.5}, L_{41}^{5.5}) & \dashrightarrow & L_{23}^{1.5} = L_{41}^{5.5}                                    \\
(L_{34}^0, 0)               & \longmapsto & (L_{34}^0, L_{12}^7) = (L_{34}^{1.5}, L_{12}^{5.5}) & \dashrightarrow & L_{34}^{1.5} = L_{12}^{5.5}                                    \\
(L_{13}^0, 0)               & \longmapsto & (L_{13}^0, L_{31}^7) = (L_{13}^{1.5}, L_{31}^7)     & \dashrightarrow & L_{13}^{1.5} = L_{31}^7                                        \\
(L_{24}^0, 0)               & \longmapsto & (L_{24}^0, L_{42}^7) = (L_{24}^0, L_{42}^{5.5})     & \dashrightarrow & L_{24}^0     = L_{42}^{5.5}                                    \\
(0, L_{jk}^{3.5})           & \longmapsto & (L_{jk}^{1.5}, L_{jk}^{5.5})                        & \dashrightarrow & L_{jk}^{1.5} = L_{jk}^{5.5}                                    \\
(0, (L_{13}-L_{24})^{3.5})  & \longmapsto & (L_{13}^{1.5}-Z_{24}^2, Z_{13}^5-L_{24}^{5.5})      & \dashrightarrow & Z_{13}^5 + Z_{24}^2 = L_{13}^{1.5} + L_{24}^{5.5}                \\
\end{array}
\end{equation}
}for $i=1,2,3,4$ and $1 \le j \ne k \le 4,$ $\{ j, k \} \ne \{ 1, 3
\}, \{ 2, 4 \}.$ In the sequence, the author uses the dashed line
$\dashrightarrow$ to mean that an element of $H_2 (A) \oplus H_2
(A^\prime)$ gives a relation in $H_2 (Y).$ For example,
$(L_{13}^{1.5}, L_{31}^7)$ gives the relation $L_{13}^{1.5} =
L_{31}^7.$ First, the rank of $(i_*, j_*)$ is 9. We explain for
this. It is easy that four $(L_{jk}^{1.5}, L_{jk}^{5.5})$'s,
$(L_{13}^{1.5}, L_{31}^7),$ $(L_{24}^0, L_{42}^{5.5}),$
$(L_{13}^{1.5}-Z_{24}^2, Z_{13}^5-L_{24}^{5.5})$ are independent in
$H_2 (A) \oplus H_2 (A^\prime)$ because we know basis of $H_2 (A),
H_2 (A^\prime).$ Let $M$ be the subspace of $H_2 (A) \oplus H_2
(A^\prime)$ generated by these, and let $M^\prime$ be its subspace
generated by $(L_{jk}^{1.5}, L_{jk}^{5.5})$'s. The remaining are
\begin{equation*}
(L_{12}^{1.5}, L_{34}^{5.5}), \quad (L_{14}^{1.5}, L_{32}^{5.5}),
\quad (L_{23}^{1.5}, L_{41}^{5.5}), \quad (L_{34}^{1.5},
L_{12}^{5.5}).
\end{equation*}
Since two sums
\begin{align*}
(L_{12}^{1.5}, L_{34}^{5.5}) + (L_{34}^{1.5}, L_{12}^{5.5})
&= (L_{12}^{1.5} + L_{34}^{1.5}, L_{12}^{5.5} + L_{34}^{5.5}),   \\
(L_{14}^{1.5}, L_{32}^{5.5}) + (L_{32}^{1.5}, L_{14}^{5.5}) &=
(L_{14}^{1.5} + L_{32}^{1.5}, L_{14}^{5.5} + L_{32}^{5.5})
\end{align*}
are contained in $\big\langle (L_{jk}^{1.5}, L_{jk}^{5.5})
\big\rangle \subset M,$ we only have to consider
\begin{equation*}
(L_{12}^{1.5}, L_{34}^{5.5}), \qquad (L_{14}^{1.5}, L_{32}^{5.5}).
\end{equation*}
Let $M^{\prime \prime}$ be the subspace of $H_2 (A) \oplus H_2
(A^\prime)$ generated by these two. Then, $M^{\prime \prime} \cap M
= M^{\prime \prime} \cap M^\prime$ because $L_{24}^0, L_{31}^7,
Z_{24}^2, Z_{13}^5$ does not appear in any component of elements of
$M^{\prime \prime}.$ It is easy that $M^{\prime \prime} \cap
M^\prime = \langle 0 \rangle.$ So, the rank of $(i_*, j_*)$ is 9.
This implies that the rank of $k_*-l_*$ is 5 by exactness of the
sequence. Since the image of $k_*-l_*$ is generated by
\begin{equation*}
\big\{~ L_{12}^0,~ L_{14}^0,~ L_{13}^0,~ L_{24}^0,~ Z_{24}^2 \text{
or } Z_{13}^5 ~\big\}
\end{equation*}
by relations (\ref{relation: n=2 l=1,2}), (\ref{relation: n=2
l=5,6}), this is a basis of $\im k_*-l_*$ because its dimension is
5. Recall that $\ker (i_*, j_*)$ for $n=1$ is equal to
\begin{equation*}
\big\langle (L_1^0 + L_3^0, -L_1^{3.5} - L_3^{3.5}), ~ (L_2^0 +
L_4^0, -L_2^{3.5}-L_4^{3.5}), ~ (L_\fiber^0, 0) \big\rangle ~ \cong
~ \R^3
\end{equation*}
by (\ref{equation: ker of i, j for l=1,2,5,6}). We have
\begin{equation*}
\begin{array}{ccc}
\partial_* (T_{1+3}^2)   & = &  (L_1^0 + L_3^0, -L_1^{3.5} - L_3^{3.5}),  \\
\partial_* (T_{2+4}^2)   & = &  (L_2^0 + L_4^0, -L_2^{3.5} - L_4^{3.5}),  \\
\partial_* (G_{61})      & = &   (L_\fiber^0, 0).       \\
\end{array}
\end{equation*}
That is,
\begin{equation*}
\im \partial_* = \langle ~ \partial_* (T_{1+3}^2), ~ \partial_*
(T_{2+4}^2), ~ \partial_* (G_{61}) ~ \rangle.
\end{equation*}
Since we already know that the rank of $\im \partial_*$ is equal to
3, we can conclude that {\footnotesize
\begin{equation} \label{homology: n=2 l=1,2,5,6}
H^2(W) = \langle ~ L_{12}^0, ~ L_{14}^0, ~ L_{13}^0, ~ L_{24}^0 ~
\rangle + \langle ~ Z_{24}^2 \text{ or } Z_{13}^5 ~ \rangle +
\langle ~ T_{1+3}^2, ~ T_{2+4}^2, ~ G_{61} ~ \rangle \cong \R^8
\end{equation}
}by exactness. This is the proof of Proposition \ref{proposition:
H_2}.

Last, we prove Proposition \ref{proposition: calabi-yau}, i.e. the
main theorem.
\begin{proof}[Proof of Theorem \ref{theorem: main}]
We only have to show that $c_1 (TW|_{T_{1+3}^2})=0$ and $c_1
(TW|_{T_{2+4}^2})=0$ because $c_1 ( T W|_Q)=0$ for other generators
of $H^2(W)$ is easy. The tangent spaces of $X$ restricted to
$L_{1+3} \times [0, 7]$ is trivial. Then, $TW|_{T_{1+3}^2}$ is
constructed by gluing of $(L_{1+3} \times [0, 7]) \times \C^3$
through the map
\begin{equation*}
\begin{array}{rcl}
(L_{1+3} \times 0) \times \C^3  & \longrightarrow  & (L_{1+3} \times
7) \times \C^3, \\
\big( (x, 0), (z_1, z_2, z_3) \big) & \longmapsto & \big( (x, 7),
(z_2, z_1, z_3) \big).
\end{array}
\end{equation*}
This gluing map is just a writing of $\tau$ by using complex
coordinates. The bundle $TW|_{T_{1+3}^2}$ has three subbundles
$\eta_i$ for $i=1,2,3$ whose pullbacks $\tilde{\eta}_i =
\tilde{\pi}^* \eta_i$ to $L_{1+3} \times [0, 7]$ are as follows:
\begin{align*}
\eta_1 &= \{(x, s, z, z, 0) \in (L_{1+3} \times [0, 7]) \times \C^3
~|~ x \in L_{1+3}, ~ s \in [0, 7], ~ z \in \C \}, \\
\eta_2 &= \{(x, s, 0, 0, z) \in (L_{1+3} \times [0, 7]) \times \C^3
~|~ x \in L_{1+3}, ~ s \in [0, 7], ~ z \in \C \}, \\
\eta_3 &= \{(x, s, z, -z, 0) \in (L_{1+3} \times [0, 7]) \times \C^3
~|~ x \in L_{1+3}, ~ s \in [0, 7], ~ z \in \C \}.
\end{align*}
Then, $\tilde{\eta}_1$ and $\tilde{\eta}_2$ are easily trivial. The
subbundle $\eta_3$ has a nonvanishing section
\begin{equation*}
(x, s) \mapsto \Big( x, s, \exp( \pi s \imath /7), -\exp( \pi s
\imath /7), 0 \Big).
\end{equation*}
And, this gives a nonvanishing section of $\tilde{\eta}_3.$ So,
$\tilde{\eta}_3$ is trivial, and we can conclude that $c_1
(TW|_{T_{1+3}^2})=0.$ Similarly, $c_1 (TW|_{T_{2+4}^2})=0.$
\end{proof}

\end{document}